\documentclass[12pt,a4paper,reqno]{article}
 
\usepackage{amsmath,amssymb,amsthm,amscd,dsfont,psfrag,color,mathrsfs} 
 
\usepackage[english]{babel}
\usepackage[mathscr]{eucal}
\usepackage[dvips]{graphicx}
 
\font\sc=rsfs10 at 12pt
 
\input{comment.sty}
\includecomment{task}
\includecomment{MM}
\includecomment{JK}
\numberwithin{equation}{section}
 
\renewcommand{\a}{\alpha}
\renewcommand{\b}{\beta}
\newcommand{\g}{\gamma}

\renewcommand{\d}{\delta}
\newcommand{\D}{\Delta}

\newcommand{\ve}{\varepsilon}

\renewcommand{\th}{\theta}

\renewcommand{\k}{\kappa}

\renewcommand{\l}{\lambda}

\newcommand{\x}{\xi}

\newcommand{\vf}{\varphi}

\newcommand{\h}{\chi}

\renewcommand{\o}{\omega}

\newcommand{\C}{{\mathbb C}}
\newcommand{\R}{{\mathbb R}}
\newcommand{\Z}{{\mathbb Z}}

\newcommand{\mbf}[1]{\protect\mbox{\boldmath$#1$\unboldmath}}

\newcommand{\Bc}{{\mathcal B}}

\newcommand{\Dc}{{\mathcal D}}

\newcommand{\Oc}{{\mathcal O}}

\newcommand{\Rc}{{\mathcal R}}
\newcommand{\Sc}{{\mathcal S}}

\newcommand{\pdf}[2]{\frac{\partial {#1}}{\partial {#2}}}

\DeclareMathOperator*{\slim}{s-lim} 
\DeclareMathOperator{\im}{{\rm Im}\,}
\DeclareMathOperator{\re}{{\rm Re}\,}
\DeclareMathOperator{\rank}{rank}

\newcommand{\Tr}{\operatorname{Tr\,}}
\newcommand{\tr}{\operatorname{tr\,}}

\newcommand{\supp}{\operatorname{supp\,}}

\newcommand{\spec}{\operatorname{spec\,}}

\newcommand{\smac}{\operatorname{spec}_{\operatorname{ac}}}

\newcommand{\sme}{\operatorname{spec}_{\operatorname{ess}}}

\newcommand{\ham}[1]{\mathbb{#1}} 
 
\newcommand{\ccs}{C_{0}^{\infty}} 
 
\newcommand{\Fr}{\sc\mbox{F}\hspace{0.1pt}} 

\newtheorem{theorem}{Theorem}[section]
\newtheorem{proposition}[theorem]{Proposition}
\newtheorem{lemma}[theorem]{Lemma}

\theoremstyle{definition}
\newtheorem{definition}[theorem]{Definition}

\theoremstyle{remark}

 {\begin{list}{}%
         {\setlength{\leftmargin}{#1}}%
         \item[]%
 }
 {\end{list}}

\author{J. Kungsman \\ 
             Department of Mathematics \\
             Uppsala University \\
             SE-751 06 Uppsala, Sweden
             \and 
             M. Melgaard \footnote{First author: M. Melgaard, Tel: +44 1273 678933, Fax:  +44 (0)1273 678097}  \\
             Department of Mathematics \\
             University of Sussex \\
             Brighton BN1 9QH, Great Britain}

\date{January 29, 2014}

\title{Poisson wave trace formula \\Êfor perturbed Dirac operators}
\theoremstyle{plain} 
\theoremstyle{plain} 

\begin{document}
\maketitle

\begin{abstract}
We consider self-adjoint Dirac operators $\ham{D}=\ham{D}_0 + V(x)$, where $\ham{D}_0$ is the free three-dimensional Dirac operator and $V(x)$ is 
a smooth compactly supported Hermitian matrix. We define resonances of $\ham{D}$ as poles of the meromorphic continuation of its cut-off resolvent. 
An upper bound on the number of resonances in disks, an estimate on the scattering determinant and the Lifshits-Krein trace formula then leads to a 
global Poisson wave trace formula for resonances of $\ham{D}$.    
\end{abstract}


\section{Introduction and main result}
\label{jkmm13b:intro} 
For suitable perturbations $L$ of the Laplacian $-\D $, on $L^2(\R ^n)$ with $n$ odd, the first Poisson wave trace formula (in the sense of distributions) of the form
\begin{equation}
2\Tr \big (\cos t\sqrt{L} - \cos t\sqrt{-\D}  \big ) = \sum e^{it\l _j}
\label{jkmm13b:intro-eq1} 
\end{equation}
where the sum extends over all resonances $\l _j$ of $L$, appeared in Lax and Phillips \cite{lax_phillips_78}. They proved the formula for obstacle scattering, 
but only for $t > 4R$ where the obstacle is located within a ball of radius $R$.  Bardos, Guillot and Ralston \cite{bardos_guillot_ralston} investigated the precise 
distribution of scattering poles associated with perturbations of the wave equation in odd dimensions and one of their main ingredients is an extension of the 
trace formula (\ref{jkmm13b:intro-eq1}), still based on Lax-Phillips theory \cite{lax_phillips_78}, but again only valid for comparatively large values of $t$, namely $t > 2R$.  
Later the Poisson formula (\ref{jkmm13b:intro-eq1}) was proved for all $t \neq 0$ by Melrose $\cite{melrose_82}$ in the case of compactly supported potentials and 
then further generalized by Sj\"{o}strand and Zworski \cite{sjo_zwo_94} to more general $L$.  

The main application of the Poisson trace formula has been to obtain lower bounds for the number of resonances in certain regions and to prove the existence of 
infinitely many resonances, see e.g. \cite{melrose_geometric_95}, \cite{barreto_zworski_96}, \cite{sjo_zwo_93} and \cite{sjo_zwo_94}. A survey of results related to 
trace formulas and resonances for Schr\"odinger type operators can be found in Burq \cite{burq98} and in the recent review by Hislop \cite{hislop_survey}.  

Zworski \cite{zworski_poisson_97} obtained a new proof of the trace formula that avoids the use of Lax-Phillips theory and instead is based on an estimate of the 
scattering determinant in $\C $.  This proof allowed Zworski to show the trace formula also in the even-dimensional case, see \cite{zworski_even_98}, and his proof 
also motivated the present work.

Proceeding to the relativistic setting, the only trace formula for Dirac operators involving resonances that we know of is that of Khochman \cite{khochman07}, where 
a local trace formula for resonances in the spirit of Sj\"{o}strand \cite{sjostrand_97} is established.

In the present work we consider Dirac operators $\ham{D} = \ham{D}_0 + V(x)$ where $\ham{D}_0$ is the free Dirac operator in three dimensions (see below) and 
$V$ is a smooth compactly supported matrix potential. We define resonances as poles of the meromorphic continuation of the cut-off resolvent. Then, following the 
strategy of Zworski \cite{zworski_even_98} (see also \cite{zworski_poisson_97}, \cite{guillope_zworski} and \cite{zworski_notes}) we establish an upper bound for the resonance 
counting function, estimate the scattering determinant and apply the Lifshits-Krein trace formula to obtain, in the distributional sense, the following Poisson wave trace formula
for the perturbed Dirac operator:

\begin{theorem}
Let $\Rc $ denote the set of resonances of $\ham{D}$ and let $m_j$ be the multiplicity of a resonance $\l _j$ (see Section~\ref{jkmm13b:diracreso} for precise definitions). 
Assume also that $\pm 1 \not \in \Rc $. Then, in the sense of distributions on $\R \setminus \{0\}$, 
\begin{align}
2 \Tr (\cos (t\ham{D}) - \cos (t \ham{D} _0)) &=  \sum _{\l _j \in  \Rc \cap \C _{+}} m_j(e^{-i|t|\overline{\l _j}} + e^{i|t|\l _j}) \nonumber \\
&\phantom{oooooooooo} - \sum _{\l _j \in  \Rc \cap \C _{-}} m_j (e^{-i|t|\l _j} + e^{i|t|\overline{\l _j}}) \nonumber \\
&\phantom{oooooooooo}+  \sum _{\l _j \in \spec _{\rm{d}}(\ham{D}) } 2 m_j \cos (t\l _j) . 
\end{align}
\label{jkmm13b:thm_trace_form}
\end{theorem}


We denote various positive constants for which the exact numerical values are of no importance by $C$. These constants may change from line to line without this being indicated.


\section{Preliminaries}
\label{jkmm13b:prelim}

\subsection{The Dirac operator}
\label{jkmm13b:diracope} 
To discuss perturbed Dirac operators we begin by considering the free, or unperturbed, Dirac operator. The free Dirac operator, describing the motion of a relativistic electron or 
positron without external forces, is the unique self-adjoint extension of the symmetric operator 
$$
\ham{D}_0 = -i \sum _{j=1}^3 \a _j \partial _j + \b = -i\mbf{\a }\cdot \nabla + \b  , \qquad \partial _j := \pdf{}{x_j}, \; \mbf{\a } = (\a _1, \a _2, \a _3),
$$
defined on $\ccs (\R ^3;\C ^4)$ in the Hilbert space $L^{2}(\R ^3 ; \C ^4)$. Here the $\a _j$ are symmetric $4\times 4$ matrices satisfying the anti-commutation relations
$$
\a _j \a _k + \a _k \a _j = 2\d _{jk}I_4, \quad j,k = 1,2,3,4, 
$$ 
with $\a _4 = \b )$ and with $I_4$ denoting the $4\times 4$ identity matrix.  The extension, which we also denote by $\ham{D}_0$, acts on the Hilbert space 
$L^2(\R^3 ; \C ^4)$ equipped with the inner product
$$
\langle u, v \rangle _{L^2(\R ^3 ; \C ^4)} = \sum _{j=1}^4 \int \limits _{\R ^3} u_j(x) \overline{v_j(x)}\,dx \quad \text{where }u=	(u_j)_{1\le j\le 4}, \; v = (v_j)_{1\le j\le 4}
$$
and it has domain $H^1(\R ^3 ; \C ^4)$; the Sobolev space of order one. When there is no risk of confusion we sometimes just write $L^2$ and $H^1$, respectively. 
It is well-known (see, e.g.,  Thaller \cite{thaller}) that the spectrum of $\ham{D}_0$ is purely absolutely continuous, viz. 
$$
\spec(\ham{D}_0) = \smac(\ham{D}_0) = (-\infty , -1] \cup [1, \infty ).
$$
On the resolvent set $\C \setminus \spec (\ham{D}_0)$ we denote the free resolvent $(\ham{D}_0  - \l )^{-1}$ by $R_0(\l )$. As usual the Fourier transform is defined by 
$$
(\Fr u) (\xi )=(2\pi )^{-3/2} \int \limits _{\R ^3 } e^{-ix\cdot \xi } u(x) \, dx. 
$$
The (principal) symbol of the free Dirac operator $\ham{D}_{0}$ is given by 
$$
\mbf{d}_0(\xi ) =\Fr \ham{D}_0 \Fr^{\ast} = \sum _{j=1}^3 \a _j \x _j + \b 
$$
and it has two doubly degenerate eigenvalues $\pm \sqrt{\x ^2 + 1} =: \pm \langle \x \rangle$. 
The corresponding orthogonal projections onto the eigenspaces are given by 
\begin{align}
	\Pi _{\pm } (\xi ) = \frac{1}{2} \big (I_4 \pm \langle \xi \rangle ^{-1}\mbf{d}_0(\xi ) \big ).
	\label{jkmm13b:projections_eigenspaces_symbol}
\end{align}
We are going to consider perturbations of $\ham{D}_0$ by smooth compactly supported Hermitian $4\times 4$ matrix potentials 
$V \in \ccs (\R ^3) \otimes M_4(\C )$; $M_{4}(\C)$ being the set of $4 \times 4$ matrices over $\C$, equipped with the operator norm, 
designated by $\| \cdot \|_{4 \times 4}$. The resulting self-adjoint operator $\ham{D} = \ham{D}_0 + V$ (defined via Kato-Rellich's theorem) 
has domain $H^{1}(\R^{3}; \C^{4})$ and, according to Weyl's theorem,  $\sme(\ham{D}) = (-\infty , -1] \cup [1, \infty )$ but, in addition, 
$\ham{D}$ can have finitely many eigenvalues of finite multiplicity in $(-1,1)$ (see, e.g.,  \cite[Theorem 4.23]{thaller}). It is well-known 
that under our assumptions on $V$ there are no eigenvalues $\l $ with $|\l |>1$ embedded in the continuous spectrum 
(see, e.g., \cite{berthier_georgescu}). 
%
%

When $A:L ^2(\R ^3; \C ^4) \to L ^2(\R ^3; \C ^4)$ is a compact operator the eigenvalues of $(A^\ast A)^{1/2}$, indexed in non-increasing order, 
are called the singular values of $A$ and are denoted by $s_j(A)$. The following inequalities are well-known (see e.g. \cite{simon_trace}):
\begin{align}
s_j(AB)&\le \|B \| s_j(A), \quad A\in \Bc _\infty , B\in \Bc , \label{jkmm13b:bounded_compact_singular_values} \\
s_{j+k-1}(A+B) &\le  s_j(A) + s_k(B) , \quad A,B \in \Bc _\infty , \label{jkmm13b:ky_fan_sum} \\ 
s_{j+k-1}(AB) &\le  s_j(A)s_k(B), \quad A, B \in \Bc _\infty , \label{jkmm13b:ky_fan_product}
\end{align}
where $\Bc $ and $\Bc _\infty $ denote the spaces of bounded and compact operators on $L^2(\R ^3; \C ^3)$, respectively. 
As a consequence of Weyl's inequality we also have 
\begin{align}
|\det (I - A)| \le \prod _{j=1}^\infty \big (1 + s_j(A) \big) \quad \text{for }A\in \Bc _1,
\label{jkmm13b:weyl_inequality}
\end{align}
where $\Bc _1$ is the set of trace class operators. 
%
%

\subsection{Resonances}
\label{jkmm13b:diracreso} 
By virtue of \eqref{jkmm13b:distribution_kernel_free_resolvent} the free cut-off resolvent $\h R_0(\l ) \h $, $\h \in \ccs (\R ^3)$, is an integral operator with kernel given by  
\begin{multline*}
	\h (x)R_0(\l ,x, y) \h (y) \\
	= \h (x) \Big (i \frac{\mbf{\a }\cdot (x-y)}{|x-y|^2} + \k (\l )\frac{\mbf{\a }\cdot (x-y)}{|x-y|} + \b + \l \Big )\frac	{e^{i\k (\l )(x-y)}}{4\pi |x-y|}\h (y),
\end{multline*}
for $\k (\l ) := \sqrt{\l ^2 -1}$, on the branch with $\im \sqrt {\l ^2 -1} > 0$.
Thus $\h R_0(\l ) \h : L ^2(\R ^3; \C ^4) \to L^2 (\R ^3; \C ^4)$ has a holomorphic extension from
$$
	\{\re \l \le 1, \im \l <0  \}\cup \{\re \l \ge -1, \im \l >0  \}
$$
across $(-\infty , -1]\cup [1,\infty )$ to the sheet with $\im \sqrt{\l ^2 - 1}<0$. 

We next consider the full resolvent $R_V(\l ) := (\ham{D} - \l )^{-1}$ for $\h \in \ccs (\R ^3)$ such that $\h V =V$. If we take $|\im \l | $ so large that $\| V R_0(\l ) \h \| \le C|\im \l |^{-1}\le 1/2$ we may write  
$$
	R_V(\l ) = R_0(\l ) (I + V R_0(\l ))^{-1}.
$$
Notice that for such $\l $ we have 
$$
	(I - VR_0(\l ) (1-\h )) (I + VR_0(\l )) = I + VR_0(\l ) \h , 	
$$ 
since $(1- \h )V =0$ and, consequently,
\begin{align}
R_V(\l ) = R_0 (\l ) \big (I + VR_0(\l ) \h  \big )^{-1} \big (I - VR_0(\l ) (1 - \h ) \big ).
\label{jkmm13b:full_resolvent_perturbation_free_resolvent}
\end{align} 
Since $VR_0(\l ) \h :L^2(\R ^3; \C ^4) \to H^1 (\supp V ; \C ^4)$, we infer that $VR_0(\l ) \h$ is compact on $L^2(\R ^3; \C ^4)$  by the 
Rellich--Kondrachov theorem and, moreover,  it depends holomorphically on $\l $ since $R_0(\l )$ does. Thus $(I + VR_0(\l ) \h )^{-1}$ 
has a meromorphic extension by the analytic Fredholm theorem. Furthermore, from \eqref{jkmm13b:full_resolvent_perturbation_free_resolvent} 
we get for $\h _0 \in \ccs (\R ^3)$ with $\h _0 V = V$ and $\h _0 \h = \h _0$ that
\begin{align}
\h _0 R_V (\l ) \h _0 = \h _0 R_0 (\l ) \h _0 \big (I + VR_0(\l ) \h  \big )^{-1} \big (I - VR_0(\l ) (1 - \h ) \big )
\label{jkmm13b:cutoff_resolvent_rewrite}
\end{align}
since initially $(I + VR_0(\l ) \h )^{-1}\h _0  = \h _0 (I + VR_0(\l ) \h )^{-1}$ holds for $|\im \l | \gg 1$ in the physical sheet by considering 
$(I + VR_0(\l ) \h )^{-1}$ as a Neumann series and remains true because both sides have meromorphic extensions. This provides the 
meromorphic extension of $\h _0 R_V(\l ) \h _0$ for which the poles are those of $(I + VR_0(\l ) \h )^{-1}$. These poles will be referred 
to as resonances of $\ham{D}$ and the set of all resonances of $\ham{D}$ will be denoted by $\Rc $. 
\begin{definition}
Assume $\l _j \in \Rc $ and let $\g _{\l _j}$ be the circle $\l _j + \ve _je^{i[0,2\pi ]}$ where $\ve _j$ is chosen sufficiently small so that $\g _{\l _j}$ 
encircles no other resonances but $\l _j$. The multiplicity of $\l _j$ is then given by 
\begin{align}
m_j = \rank \int \limits _{\g _j} R(\l ) \, d \l .
\label{jkmm13b:def_multiplicity_resonance}
\end{align}
\end{definition}


The proof of the following upper bound, which goes back to Melrose \cite{melrose83} in the Schr\"{o}dinger case, for the number of resonances in disks 
follows a by now standard procedure 
(see,  e.g.,  \cite{hislop_survey}, \cite[Proposition~6.2]{jkmm12a}, \cite{zworski_sharp}, \cite{vodev_sharp}, and references therein). 

\begin{proposition}
The following upper bound holds true: 
\begin{align}
N(r) :=  \# \{ \l \in \Rc : |\l | \le r \} \le Cr^3, 
\label{jkmm13b:resonance_counting_function}
\end{align}
where the number of resonances are counted according to their multiplicities. 
\label{jkmm13b:upper_bound_resonances}
\end{proposition}


\begin{proof}
Recall, by \eqref{jkmm13b:full_resolvent_perturbation_free_resolvent}, that the resonances of $\ham{D}$ can be characterized as poles of 
$(I + VR_0 (\l ) \h )^{-1}$. Since 
$$
I - (VR_0(\l ) \h )^4 = (I + VR_0 (\l ) \h ) \big (I - VR_0 (\l ) \h  + \cdots - (VR_0 (\l ) \h )^3 \big ) 
$$
and because (see below) $(VR_0(\l ) \h )^4\in \Bc _1$, the resonances will appear among the zeros of 
$$
f(\l ) = \det \big (I - (VR_0(\l ) \h )^4 \big ). 
$$
We first estimate $|f(\l )|$ by using the Weyl inequality \eqref{jkmm13b:weyl_inequality}:
\begin{align}
|f(\l )| \le \prod _{j=1} ^\infty \Big (1 + s_j \big ((VR_0(\l ) \h )^4 \big ) \Big ). 
\label{jkmm13b:weyl_inequality_applied}
\end{align}
In view of Ky Fan's inequality \eqref{jkmm13b:ky_fan_product} and the inequality $s_j(VR_0(\l ) \h ) \le \|V \| _\infty s_j(\h  R_0(\l ) \h)$ 
it suffices to estimate the singular values of $\h  R_0(\l ) \h$. This can be done by comparing them 
to the singular values of the resolvent of a free Dirac operator on a sufficiently large flat torus $\ham{T} = (\R / R \Z )^3$ with 
$\supp (\h  )\subset B(0,R)$:
$$
s_j(\h  R_0(\l ) \h) \le s_j \big ((\ham{D}_{\ham{T},0} - i)^{-1} \big ) \| \h  R_0(\l ) \h \| _{L ^2 \to H ^1}.
$$ 
It is well-known (see e.g. \cite{jkmm12a}) that $s_j((\ham{D}_{\ham{T},0} - i)^{-1}) \le Cj^{-1/3}$. Moreover, on the branch of the square root 
where $\im (\k (\l )) > 0$ we obtain from \eqref{jkmm13b:dirac_resolvent_laplace_resolvent} and 
\eqref{jkmm13b:laplace_resolvent_bound_sobolev} that $\| \h  R_0(\l ) \h \| _{L^2 \to H^1} \le C \langle \l \rangle $. 

To estimate the singular values of the extended resolvent we use 
$$
\h \big (R_0(\l  ) - \widetilde{R}_0(\l )  \big ) \h = 2\pi i E_{\h } (\overline{\l }) ^\ast E_{\h}(\l ),
$$
(see \eqref{jkmm13b:def_E_chi} and \eqref{jkmm13b:extended_resolvent_relation}), where we temporarily denote the extended resolvent by 
$\widetilde{R}_0(\l )$. Therefore, it suffices to estimate 
$$
s_j(E_{\h} (\overline{\l }) ^\ast E _{\h}(\l )) \le \| E _{\h}(\l ) \|_{L^2\to L^2(S^2)} s_j(E_{\h} (\l )). 
$$
To this end, denote by $\D _\o $ the Laplace-Beltrami operator on $S ^2$. Then   
$$
  s_j(E_{\h}(\l )) \le s_j\big ( (I - \D _{\o })^{-k} \big ) \|(I - \D _{\o })^{k}E_{\h}(\l )  \|_{L ^2 (\R ^3) \to L ^2 (S^2)}.
$$
where 
$$
  s_j\big ( (I - \D _{\o })^{-k} \big ) \le C^k j^{-k}
$$
is well-known, and
\begin{align*}
  \|(I - \D _{\o })^{k}E_{\h}(\l )  \|_{L ^2(\R ^3) \to L ^2(S^2)} &\le (2k)! e^{C |\l |}.
\end{align*}
Hence, for $k=[j^{1/2}/(2C)]+1$,  
$$
s_j(E_{\h}(\l ) ) \le j^{-k}C^k (2k)! e^{C |\l |} \le C_1 e^{C_1 |\l | - j^{1/2}/C_1},
$$
by Stirling's formula.  Since $\| \ham {E}_{\h } (\l ) \| _{L^2 \to L^2(S^2)} \le Ce^{C|\l |}$ we now get 
$$
s_j \big ((VR_0(\l ) \h )^4 \big ) \le C_2\exp{(C_2 |\l | - j^{-1/2}/C_2)} + C_2 j^{-4/3}, 
$$
so that, especially, 
$$
s_j \big ((VR_0(\l ) \h )^4 \big ) \le \begin{cases} C_3 \exp{(C_3|\l |)} \quad &\text{for } j\le C_3|\l |^2  , \\
C_3 j^{-4/3} \quad &\text{for } j\ge C_3 |\l |^2. \end{cases}
$$
It follows from \eqref{jkmm13b:weyl_inequality} that 
\begin{align}
|f(\l )| &\le \prod _{j\le C_3 |\l |^2}\Big (1 + C_3 \exp {(C_3|\l |)} \Big ) \Big ( \exp {(\sum _{j\ge C_3 |\l|^2} C_3 j^{-4/3} ) } \Big ) \\
&\le \exp {({C_4 |\l |^3})}.
\label{jkmm13b:determinant_idpluspowercutoffresolvent_estimate}
\end{align}
The result now follows by an application of Jensen's formula
\begin{equation*}
\log(2)N(r) \le \frac{1}{2\pi } \int \limits _0 ^{2\pi } \log |f(2re^{i\th })| \, d \th - \log |f(0)|. \qedhere 
\end{equation*}
\end{proof}
%
%
%


\section{Proof of the Poisson trace formula}
\label{jkmm13b:provetrace}
Let $s(\l)$ be the determinant of the scattering matrix, as introduced in Appendix A. Meromorphic extension of the identity 
$S(\l )^{-1} = S(\l ) ^\ast $ for $\l \in \spec (\ham{D}_0)$ implies $S(\l )^{-1} = S(\overline{\l }) ^\ast $ and,  therefore,  the scattering determinant 
satisfies
$$
	\frac{1}{s(\l )} = \overline{s(\overline{\l })}. 
$$ 
It follows that if $z_j\in \Rc $ then $s(\overline{z_j}) = 0$ and vice versa.   
From the Weierstrass factorization theorem it follows that
\begin{align}
s(\l ) = e^{g(\l )} \frac{P_{\overline{\Rc  }}(\l )}{P_{\Rc }(\l )},
\label{jkmm13b:weierstrass_factorization}
\end{align}
where $P_{\Rc }$ is the canonical product
\begin{align}
P_\Rc (\l ) = \prod _{z_j \in \Rc } E_3\Big (\frac{\l }{z_j} \Big )^{m_j}, \quad E_p (z ) = (1-z) \exp {\Big ( z + \frac{z^2}{2} + \cdots + \frac{z^p}{p}} \Big )
\label{jkmm13b:canonical_product}
\end{align}
and $g$ an entire function. 
We choose the genus to be $p=3$ so the infinite product converges (see e.g. \cite{conway}). 
Moreover,  it follows from  \cite[Appendix~D.2, equation (D.2.2)]{zworski_notes}
together with Proposiion~\ref{jkmm13b:upper_bound_resonances} that
\begin{align}
|P_\Rc (\l )| \le C e^{C|\l |^4}.
\label{jkmm13b:weierstrass_product_upper_bound}
\end{align} 
\begin{lemma}
For any constant $C>0$ there exists a constant $C_0$ such that the inequality
$$
\prod _{j=1}^\infty (1 + Ce^{C|\l |^N - \sqrt{j}/C}) \le C_0e^{C_0|\l | ^{3N}}
$$
holds for all $\l \in \R $.
\label{jkmm13b:infinite_product_j_estimate}
\end{lemma}
\begin{proof}
Clearly 
$$
\sum _{j=1}^\infty \log (1 + Ce^{C|\l |^N - \sqrt{j}/C}) \le \int \limits _{0}^\infty  \log (1 + Ce^{C|\l |^N - \sqrt{x}/C}) \, dx =:L,
$$
and an upper bound of the product is given by $e^L$. Integrating by parts and denoting $\varphi _0 (\l ) = 2(\log C + C|\l |^{N} )$ we get
\begin{multline*}
L = C^2 \int \limits _{0}^\infty x^2 \frac{Ce^{C|\l |^N }}{e^x + Ce^{C|\l |^N }} \, dx  \le  C^2 \int \limits _{0}^{\varphi _0 (\l )} x^2 \, dx + C^2 \int \limits _{\varphi _0 (\l )}^{\infty } \frac{x^2}{e^{x/2} + 1}\, dx \\
\le C_1 + C_1|\l |^{3N}
\end{multline*}
and the result follows. 
\end{proof}
Next we estimate the determinant of the scattering matrix. The proof of the following lemma is in the spirit of Zworski's work \cite{zworski_poisson_97,zworski_even_98} 
on Schr\"{o}dinger operators; see also \cite[Proposition~6.3]{jkmm12a} and \cite[p 9]{petkovzwor99}. 

\begin{lemma}
For any $\ve , \d >0$ we have 
$$
|s(\l )| \le C e ^{|\l |^{9+3\ve }} \quad \text{for }\l  \not \in  \bigcup _{z_j\in \Rc }D(z_j, \langle z_j \rangle ^{-3-\d }).
$$
\label{jkmm13b:scattering_phase_estimate_lemma}
\end{lemma}
\begin{proof}
We begin by using the estimate 
\begin{align}
|s(\l )| \le \prod _{j=1} ^\infty \big (1 +  s_j(A(\l ))  \big ),
\label{jkmm13b:scattering_phase_sing_values_scattering_amplitude}
\end{align}
where
\begin{align}
  s_j(A(\l )) \le C \| V\big (I + \h (\ham{D}_0 - \l )^{-1} V \big )^{-1} \| \|E_{\h }(\l ) \|  s_j(E_{\h }(\l ) ).
\label{jkmm13b:sing_values_scattering_amplitude_estimate}
\end{align}
From \cite[Theorem 5.1]{gohkre69} we have 
\begin{align}
\| V\big (I + \h (\ham{D}_0 - \l )^{-1} V \big )^{-1} \| \le \frac{\det \big (I + |\h (\ham{D}_0 - \l )^{-1} V|^4 \big )}{|\det \big (I + (\h (\ham{D}_0 - \l )^{-1} V \big )^4)|}.
\label{jkmm13b:norm_determinants}
\end{align}  
For the numerator it follows from \eqref{jkmm13b:determinant_idpluspowercutoffresolvent_estimate} that we have the upper bound
\begin{align}
\det \big (I + |\h (\ham{D}_0 - \l )^{-1} V|^4 \big ) \le Ce^{C| \l | ^{3}}. 
\label{jkmm13b:determinant_upper_bound}
\end{align}
Then Cartan's minimum modulus principle for entire functions gives, for any $\ve , \d > 0$,
\begin{align}
|\det \big (I + (\h (\ham{D}_0 - \l )^{-1} V)^4 \big )| \ge C e ^{-C|\l |^{3+\ve }}  \text{ for }\l \not \in \bigcup _{z_j\in \Rc }D(z_j, \langle z_j \rangle ^{-3-\d }).
\label{jkmm13b:determinant_lower_bound_away_disks}
\end{align} 
From \eqref{jkmm13b:norm_determinants}, \eqref{jkmm13b:determinant_upper_bound} and \eqref{jkmm13b:determinant_lower_bound_away_disks} it then follows that 
$$
\| V\big (I + \h (\ham{D}_0 - \l )^{-1} V \big )^{-1} \| \le C e ^{C|\l |^{3+\ve } }  \text{ for }\l \not \in \bigcup _{z_j\in \Rc }D(z_j, \langle z_j \rangle ^{-3-\d }).
$$
Combined with \eqref{jkmm13b:scattering_phase_sing_values_scattering_amplitude} and \eqref{jkmm13b:sing_values_scattering_amplitude_estimate} this results in the upper bound 
$$
|s(\l )| \le \prod (1 + C e^{C|\l |^{3+\ve }}j^{-k}) \le \prod (1 + C e^{C|\l |^{3+\ve } - \sqrt{j}/C}) \le C e ^{C|\l |^{9+3\ve }} ,
$$
where we have taken $k=[j^{1/2}] + 1$ and the last step uses Lemma \ref{jkmm13b:infinite_product_j_estimate}. 
\end{proof}
Similarly to Zworski \cite{zworski_poisson_97,zworski_even_98} we are now ready to show one of the main ingredients of the proof of Theorem~\ref{jkmm13b:thm_trace_form}:

\begin{lemma}
The entire function $g$ in \eqref{jkmm13b:weierstrass_factorization} is a polynomial of degree $\le 9$. 
\label{jkmm13b:g_is_polynomial}
\end{lemma}
\begin{proof}
Applying the maximum principle to the holomorphic function $f(\l ) = s(\l ) P _\Rc  (\l ) = e^{g(\l )} P_{\overline{\Rc }}(\l )$ we get from Lemma \ref{jkmm13b:scattering_phase_estimate_lemma} and \ref{jkmm13b:weierstrass_product_upper_bound} that
$$
  |f(\l )| = |e^{g(\l )}| |P_{\overline{\Rc }}(\l )| \le C e ^{C|\l |^{9+3\ve }}. 
$$
Away from $\Rc $ we have (see \cite{conway}, Chap. XI, \S 3, Lemma 3.1) 
$$
\frac{d^N}{d\l ^N} \Big ( \frac{f'(\l )}{f(\l )} \Big ) = -N! \sum ( \overline{z_j} - \l )^{-N+1}, \quad \text{ for } N\ge 9
$$
From 
$$
\frac{f'(\l )}{f(\l )} = g'(\l ) - \frac{P'_{\overline{\Rc }}(\l  )}{P_{\overline{\Rc }}(\l )}
$$
we therefore obtain
$$
  -N! \sum (\overline{z_j} - \l )^{-N+1} = g^{(N+1)}(\l ) - \frac{d^N}{d\l ^N} \Big ( \frac{P'_{\overline{\Rc }}(\l  )}{P_{\overline{\Rc }}(\l )} \Big ).
$$
A direct calculation of the second term on the right hand side gives $g^{(N+1)} (\l ) = 0$ for $N\ge 9$. 
\end{proof}
With these preparations we are now ready to prove our main result following Zworski's strategy \cite{zworski_notes}. 

\begin{proof}[Proof of Theorem \ref{jkmm13b:thm_trace_form}]
Let $u(t) = 2 \Tr (\cos (t\ham{D}) - \cos (t\ham{D}_0))$ and $\vf \in \ccs (\R \setminus \{0\})$. Then we can write $\vf = \vf _- + \vf _+$ where $\vf _{\pm }$ are the restrictions of $\vf $ to $\R _{\pm }$. We obtain 
\begin{align*}
\langle u, \vf \rangle _{\Dc ', \Dc  } &= \sum _{\pm } \Tr (\widehat {\vf _{\pm }} (\ham {D}) + \widehat {\vf _{\pm }} (\ham {-D}) - \widehat {\vf _{\pm }} (\ham {D}_0) - \widehat {\vf _{\pm }} (-\ham {D}_0)) \\ 
&= \sum _{\pm } \Tr (f_{\pm} (\ham {D}) - f_{\pm }(\ham{D}_0)),
\end{align*}
where we have defined $f_{\pm } (\l )=  \widehat{\vf _{\pm }} (\l ) + \widehat{\vf  _{\pm }} (- \l )$. Using \eqref{jkmm13b:lifshits_krein_trace_formula_scattering_phase} we obtain
\begin{align}
\langle u, \vf \rangle _{\Dc ', \Dc  } = -\frac{1}{2\pi i}\sum _{\pm } \Big ( \int \limits _{\R} \widehat {\vf _{\pm}}( \pm \l ) \partial _{\l } (\log s(\l ) ) \, d\l + \sum _{\l _j \in \spec _{\textrm{d}}(\ham{D})}  f_{\pm }(\l _j) \Big )
\label{jkmm13b:u_acting_phi}
\end{align}
with all four sign combinations in the integral.
Now define $h_{\pm } (\l ) = \Fr [\vf _{\pm } /t^{9}](\l )$ so that $\partial ^{9}_{\l }(h_{\pm })(\l ) = \widehat {\vf _{\pm }} (\l )$, integrate by parts and use the factorization \eqref{jkmm13b:weierstrass_factorization} to obtain 
\begin{align*}
&-\frac{1}{2\pi i}\int \limits _{\R} \widehat {\vf _{+}}( \l ) \partial _{\l } (\log s(\l ) ) \, d\l  = -\frac{1}{2\pi i} \int \limits _{\R} h_{+ }( \l ) \partial _{\l } ^{10} (\log s(\l )) \, d\l \nonumber \\
&= -\frac{1}{2\pi i} \sum _{\l _j \in \Rc } m_j \int \limits _{\R} h_{+ }( \l ) \Big (\frac{9!}{(\l - \overline{\l _j})^{10}} -  \frac{9!}{(\l - \l _j )^{10}} \Big )\, d\l \\
&=  \frac{9!}{2\pi i} \sum _{\l_j \in \Rc \cap \C _{+}} m_j \int \limits _{|\l - \overline{\l _j}|\ll 1} \frac{h_+(\l )}{(\l - \overline{\l _j})^{10}} \, d\l 
 \\
&\phantom{ooooooooooooooooooooooooooooooo}-\frac{9!}{2\pi i} \sum _{\l_j \in \Rc \cap \C _{-}} m_j \int \limits _{|\l - \l _j|\ll 1} \frac{h_+(\l )}{(\l - \l _j)^{10}} \, d\l \\
&= \sum _{\l _j \in  \Rc \cap \C _{+ }} m_j \widehat{\vf _+}(\overline{\l _j}) 
- \sum _{\l _j \in  \Rc \cap \C _{- }} m_j \widehat{\vf _+}(\l _j) ,
\end{align*}
where we have used the fact that $|h_{+ }(\l )| = \Oc (\langle \l \rangle ^{-\infty })$ for $ \im \l \le 0$ to deform the integral over $\R $ into the lower half-plane. We treat the remaining terms in  \eqref{jkmm13b:u_acting_phi} similarly to obtain 
\begin{align*}
&\langle u, \vf \rangle _{\Dc ', \Dc  } = \sum _{\l _j \in  \Rc \cap \C _{+}} m_j \widehat{\vf _+}( \overline{\l _j}) - \sum _{\l _j \in  \Rc \cap \C _{-}} m_j\widehat{\vf _+}(\l _j) \\
&\phantom{ooooooooooooo}-\sum _{\l _j \in  \Rc \cap \C _{-}} m_j \widehat{\vf _+}( -\overline{\l _j}) + \sum _{\l _j \in  \Rc \cap \C _{+}} m_j\widehat{\vf _+}(-\l _j) \\
&\phantom{ooooooooooooo}-\sum _{\l _j \in  \Rc \cap \C _{-}} m_j \widehat{\vf _-}( \overline{\l _j}) + \sum _{\l _j \in  \Rc \cap \C _{+}} m_j\widehat{\vf _-}(\l _j) \\
&\phantom{ooooooooooooo}+\sum _{\l _j \in  \Rc \cap \C _{+}} m_j \widehat{\vf _-}( -\overline{\l _j}) - \sum _{\l _j \in  \Rc \cap \C _{-}} m_j\widehat{\vf _-}(-\l _j) \\
&\phantom{ooooooooooooo}+ \sum _{\l _j \in \spec _{\textrm{d}} (\ham{D})} (f_{- }(\l _j) +  f_{+}(\l _j))\\
&= \langle \vf , \sum _{\l _j \in  \Rc \cap \C _{+}} m_j(e^{-i|t|\overline{\l _j}} + e^{i|t|\l _j}) - \sum _{\l _j \in  \Rc \cap \C _{-}} (e^{-i|t|\l _j} + e^{i|t|\overline{\l _j}}) \\
&\phantom{ooooooooooooooooooooooooooooooooooooooooo}+  \sum _{\l _j \in \spec _{\textrm{d}}(\ham{D})} 2 m_j \cos (t\l _j)  \rangle .
\end{align*}
which proves \eqref{jkmm13b:thm_trace_form}.
\end{proof}

\appendix

\section{Scattering theory}
It is well-known (see e.g. \cite{thaller} or \cite{yafaev_dirac}) that under the assumption that $V\in \ccs (\R ^3)$ the wave operators
$$
W_\pm = \slim _{t\to \pm \infty } e^{it\ham{D}} e^{-it\ham{D}_0}
$$
exist, are asymptotically complete and fulfill the intertwining relation $\ham{D} W_{\pm } = W_{\pm } \ham{D}_0$. The scattering operator $\ham{S} = W_+^\ast W_-$ for the pair $(\ham{D}, \ham{D}_0)$ is unitary on $L ^2(\R ^3)$ and commutes with $\ham{D}_0$ and consequently it can be represented as multiplication by the so called scattering matrix $S(\l )$.

To obtain such stationary representations of $S(\l )$ we need to discuss spectral representations of the Dirac operator. To this end we introduce $E_0(\l ):L^2 (\R ^3 ; \C ^4) \to L^2 (S^2 ; \C ^4)$ by
\begin{align}
	(E_0(\l ) f) (\o ) = (2\pi )^{-3/2} (\l ^2 (\l ^2 - 1))^{1/4} \Pi _\pm  (\k (\l )\o ) \int \limits _{\R ^3} e^{-i\k (\l ) \o \cdot x} f(x) \, dx,	
\label{jkmm13b:free_trace_E}
\end{align}
for $\pm \l >1$, with $\Pi _{\pm }$ as in \eqref{jkmm13b:projections_eigenspaces_symbol} and $\k (\l ) = \sqrt{\l ^2 - 1}$. The adjoint operator $E_0 (\l )^\ast: L^2 (S^2 ; \C ^4) \to L^2 (\R ^3 ; \C ^4)$ is then given by 
$$
	(E_0(\l )^\ast  f) (x) = (2\pi )^{-3/2} (\l ^2 (\l ^2 - 1))^{1/4}  \int \limits _{S^2} e^{i\k (\l ) \o \cdot x} \Pi _\pm  (\k (\l )\o ) f(\o ) \, d	\o.
$$
Then 
\begin{align*}
	&[E(\l ) \ham{D}_0 f](\o ) = (\l ^2 (\l ^2 - 1))^{1/4}\Pi _\pm  (\k (\l )\o ) \Fr  (\ham{D}_0 f) (\k (\l )\o ) \\
	&= (\l ^2 (\l ^2 - 1))^{1/4}\Pi _\pm  (\k (\l )\o ) \mbf{d}_0 (\k (\l )\o ) \Fr (f) (\k (\l )\o ) \\
	&= (\l ^2 (\l ^2 - 1))^{1/4} \l _{\pm }(\k (\l )\o ) \Pi _\pm  (\k (\l )\o ) \Fr (f) (\k (\l )\o ) \\
	&= \l [E(\l )f](\o ),
\end{align*}
which is to say that $E(\l ) \ham{D}_0 E(\l )^{-1} = \l $ where the right-hand side denotes multiplication by $\l $.

A representation of the scattering matrix is shown in \cite{balslev_helffer} but they use a different spectral representation that we now briefly discuss. 
We begin by introducing the so called Foldy-Wouthuysen (F-W) transform which diagonalizes $\ham{D}_0$ as in \cite{bjorken_drell} (see also \cite{balslev_helffer} and \cite{thaller}). In the $\x$-representation it is given by the unitary $4\times 4$ matrix defined by $\hat{G}(\x ) = \exp (\b (\mbf{\a } \cdot \xi ) \th (|\x |))$ where $\th (t) = (2t)^{-1}\arctan t$ for $t>0$. A direct calculation gives
\begin{align}
\hat{G}(\x ) \mbf{d}_0(\x ) \hat{G}(\x )^{-1} = (\x ^2 + 1)^{1/2} \b . 
\label{jkmm13b:conjugation_G}
\end{align}  
We then define the F-W transform as the unitary operator $G$ on $L^2(\R ^3; \C ^4)$ defined by $G=\Fr ^{-1}\hat{G}(\x ) \Fr $ and it transforms $\ham{D}_0$ into 
$$
\tilde{\ham{D}}_0 := G\ham{D}_0 G^{-1} = (-\D + 1)^{1/2}\b . 
$$ 
We now define the restrictions of the so called free trace operator (see \cite{balslev_helffer}) $T_{0}^{\pm }(\l ) : L^2 (\R^3 ;\C ^4) \to L^2(S^2 ; \C ^4)$ by 
$$
	(T_{0}^{\pm }(\l )f)(\o ) = (2\pi )^{-3/2}(\l ^2 (\l ^2 - 1))^{1/4} \int \limits _{\R ^3} e^{-i\k (\l )\o \cdot x} P_{\pm }G f(x) \, dx
$$
where $P_{\pm } = 2^{-1}(I_4 \pm \b )$ and take the free trace operator to be $T_{0}(\l ) = T_{0}^{\pm } (\l)$ depending on whether $\pm \l > 1$. Similarly to above it is easy to see that $T_{0}(\l )^{-1} \ham{D}_0 T_{0}(\l )$ is also multiplication by $\l $. 
 
In \cite{balslev_helffer} it is shown that the scattering matrix has the stationary representation 
\begin{align}
	\tilde{S}(\l ) = I - 2\pi i T_0(\l ) (V - V R(\l + i0) V ) T_0 (\l ) \quad \text{for }|\l | >1. 
\label{jkmm13b:scattering_matrix_free_trace}
\end{align}
We can relate this representation to the one given by $E(\l )$ in \eqref{jkmm13b:free_trace_E} by noting that by \eqref{jkmm13b:conjugation_G}
\begin{align*}
	&[\hat{G} (\k (\l ) \o ) E_0(\l ) f ](x) \\
	&=(\l ^2 (\l ^2 - 1))^{1/4} \hat{G}(\k (\l ) \o ) \tfrac{1}{2}\big (I_4 + \l ^{-1} \mbf{d}_0 (\k (\l ) \o ) \big ) \Fr (f) (\k (\l ) \o ) \\
	&=(\l ^2 (\l ^2 - 1))^{1/4}P_{\pm } \hat{G}(\k (\l ) \o ) \Fr (f) (\k (\l ) \o ) \\
	&=(\l ^2 (\l ^2 - 1))^{1/4} P_{\pm } (\Fr G f)(\k (\l ) \o ) \\
	&= [T_0f] (\o ).
\end{align*}
This together with \eqref{jkmm13b:scattering_matrix_free_trace} results in the representation
$$
	S(\l ) := \hat{G}(\k (\l ) \o ) ^{-1} \widetilde{S}(\l ) \hat{G}(\k (\l ) \o ) = I - 2\pi i E_0(\l ) (V - V R(\l + i0) V ) E_0 (\l )^{\ast } 
$$ 
for $|\l | > 1$. 
The scattering matrix is unitary for $|\l |>1$ with 
$$
	S(\l )^{-1} = I + 2\pi i E_0(\l ) (V - V R(\l - i0) V ) E_0 (\l )^{\ast }. 
$$
By taking $\h \in \ccs (B(0,R_0))$ with $\h V = V$ we use the identity 
$$
	V(I - (\ham {D} - \l )^{-1}V) = V\big (I + \h (\ham{D}_0 - \l )^{-1}V \big )^{-1}
$$
to rewrite the extended scattering matrix as
\begin{align}
	S(\l ) = I -2\pi i E_\h (\l ) V\big (I + \h (\ham{D}_0 - \l )^{-1}V \big )^{-1} E_\h (\l ) ^\ast .
\label{jkmm13b:scattering_matrix_repr}
\end{align}
where
\begin{align}
	(E_\h (\l ) f) (\o ) = (2\pi )^{-3/2} (\l ^2 (\l ^2 - 1))^{1/4} \Pi _\pm  (\k (\l )\o ) \int \limits _{\R ^3} e^{-i\k (\l ) \o \cdot x} \h (x)f(x) \, 	dx.
	\label{jkmm13b:def_E_chi}
\end{align}
From \eqref{jkmm13b:cutoff_resolvent_rewrite} we see that the resonances will appear as poles of the extended $S(\l )$.
It also follows that resonances appear as poles of the scattering determinant
$$
s(\l ) = \det (S(\l )) = \det (I + A(\l )). 
$$
\subsubsection{The Lifshits-Krein trace formula}
The so called spectral shift function $\x \in \Dc ' (\R )$ is a generalization of the eigenvalue counting function that makes sense also on the absolutely continuous spectrum $(-\infty , -1]\cup [1,\infty )$ where it is smooth. It is well-known (see \cite{yafaev_dirac}) that the Lifshits-Krein trace formula 
$$
\tr (f(\ham{D}) - f(\ham{D}_0)) = \int \limits _{\R } f(\l) \x ' (\l ) \, d \l , 
$$
holds for any $f\in \Sc (\R )$. Also, by the Birman-Krein formula we have 
$$
s(\l ) = e^{-2\pi i \x (\l )}
$$
for almost every $\l \in (-\infty , -1]\cup [1,\infty )$. Therefore we can choose a branch of the logarithm such that 
$$
\x (\l ) = -\frac{1}{2\pi i } \log s(\l ), \quad \text{for a.e. }\pm \l >1,
$$  
and obtain
\begin{align}
\tr (f(\ham{D}) - f(\ham{D}_0)) = -\frac{1}{2\pi i} \int \limits _{\R } f(\l )\partial _\l \big ( \log s(\l )\big ) \, d\l  + \sum _{\l _j \in \spec _{\textrm{d}} (\ham{D})} m_j f(\l _j),
\label{jkmm13b:lifshits_krein_trace_formula_scattering_phase}
\end{align}
under the assumption that $\pm 1 \not \in \Rc $.


\section{Resolvent of free Dirac operator}
From the identity $\ham{D}_0^2 = -\D + 1$ it follows that 
\begin{align}
R_0(\l ) = (\ham{D}_0 + \l ) R_{00}(\sqrt{\l ^2  -1}).
\label{jkmm13b:dirac_resolvent_laplace_resolvent}
\end{align}
where $R_{00}(z)= (-\D - z^2)^{-1}$. It is well-known that $R_{00}(z)$ is a convolution operator (see e.g. \cite{sjostrand01b}, \cite{melrose_geometric_95} and \cite{zworski_notes}) and that its kernel is given by 
$(4\pi )^{-1}|x|^{-1}e^{iz|x|}$ where $\im z> 0$. Consequently, for $\l \in \C \setminus \spec(\ham{D}_0)$, we have
\begin{align}
  [(\ham{D}_0 - \l )^{-1} u ](x) = (\mbf{\a } \cdot \nabla  + \beta + \l ) \frac{1}{4\pi }\int \limits _{\R ^3} \frac{e^{i\sqrt{\l ^2 - 1}|x-y|}}{|x-y|}u(y)\,dy, 
\label{jkmm13b:distribution_kernel}
\end{align}  
for $u\in L^2(\R ^3 ; \C ^4)$ on the branch where $\im (\sqrt{\l ^2 - 1})>0$. It is not difficult to show that the resolvent kernel of $\ham{D}_0$ on 
$\ccs(\R ^3 ; \C ^4)$ is given by   
\begin{align}
R_0(\l ,x) = \Big (i \frac{\mbf{\a }\cdot x}{|x|^2} + \sqrt{\l ^2 - 1}\frac{\mbf{\a }\cdot x}{|x|} + \b + \l \Big )\frac{e^{i\sqrt{\l^2 - 1}x}}{4\pi |x|}. 
\label{jkmm13b:distribution_kernel_free_resolvent}
\end{align}

\noindent
It is also well-known (see e.g. \cite{sjostrand01b} and \cite{zworski_notes}) that for $\h \in \ccs (\R ^3)$ the cut-off resolvent $\h R_{00}(z)\h $ can be extended 
holomorphically to all of $\C $ and that it admits the following upper bounds:
\begin{align}
\| \h R_{00} (z)\h \|_{L^2 \to H^j} \le  C (|z|^{j-1}e^{C(\im z)_-}), \qquad \text{for }j=0,1,2
\label{jkmm13b:laplace_resolvent_bound_sobolev}
\end{align}
where the constant $C$ depends only on the support of $\h $. 

If temporarily we denote the extended resolvent by $\widetilde{R}_{00}(z)$ it can be related to the standard resolvent by
$$
	\h R_{00}(z) \h - \h \widetilde{R}_{00}(-z) \h = T_{\h } (z), \quad \im z > 0 ,
$$
where $T(z)$ is the convolution operator with kernel 
$$
	T(x,y,z) = \h (x) \frac{i}{2} \frac{z}{(2\pi )^2} \int \limits _{S^2} e^{iz \o \cdot (x-y)} \,d \o \h (y) .
$$
It follows that if $\h \widetilde{R}_0(\l ) \h $ denotes the resolvent extended to $\im \k (\l ) < 0$ we have 
\begin{align*}
	&[(\h R_0(\l ) \h - \h \widetilde{R}_0(\l ) \h)f](x) \nonumber \\
	&= \h (x)[(\ham{D}_0 + \l ) \Big ( R_{00}(\k (\l ) ) - \widetilde{R}_{00}(-\k (\l )) \Big)	(\h f)] (x) \nonumber
	\\
	&= \frac{i}{2}\frac{\k (\l )}{(2\pi )^2}\h (x) (\ham{D}_0 + \l ) \int \limits _{S^2} e^{i\k (\l )\o \cdot x} \int \limits _{\R ^3} e^{-i\k (\l )\o 	\cdot y} \h (y) f(y) \, dy \nonumber 
\end{align*}
\begin{align}
	&= \frac{i}{2}\frac{\k (\l )}{(2\pi )^2}\h (x) \int \limits _{S^2} (\mbf{d}_0 (\k (\l )\o ) + \l )e^{i\k (\l )\o \cdot x } \int \limits _{\R ^3} e^	{-i\k (\l )\o 	\cdot y} \h (y) f(y) \, dy \nonumber \\
	&=\frac{i \l \k (\l )}{(2\pi )^2} \h (x) \int \limits _{S^2} \frac{1}{2} \Big ( I_4 \pm \l ^{-1} \mbf{d}_0(\k (\l ) \o ) \Big ) 
	e^{i\k (\l )\o \cdot x } \int \limits _{\R ^3} e^	{-i\k (\l )\o 	\cdot y} \h (y) f(y) \, dy \nonumber \\
	&=[2\pi i E_\h (\overline{\l }) ^\ast E_{\h } (\l )f](x).
	\label{jkmm13b:extended_resolvent_relation}
\end{align}
%


%
%

\end{document}